\documentclass[reqno]{amsart}     

\usepackage{amsmath,amssymb,latexsym,enumerate}
\usepackage[utf8]{inputenc}
\usepackage{ifthen}
\usepackage[style=american]{csquotes}
\usepackage{color}

\newtheorem{theorem}{Theorem}[section]
\newtheorem{lemma}[theorem]{Lemma}
\newtheorem{claim}[theorem]{Claim}
\newtheorem{proposition}[theorem]{Proposition}

\newtheorem{corollary}[theorem]{Corollary}
\newtheorem{conjecture}[theorem]{Conjecture}

\newtheorem{question}[theorem]{Question}
\newcommand{\cF}{\mathcal{F}}
\newcommand{\cP}{\mathcal{P}}

\begin{document}
\date{\today}
\title{Keeping Avoider's graph almost acyclic}

\author[Dennis Clemens]{Dennis Clemens}
\address{Freie Universit\"at Berlin, Institut f\"ur Mathematik,
Arnimallee 3, 14195 Berlin, Germany}
\email{d.clemens@fu-berlin.de}

\author[Julia Ehrenm\"uller]{Julia Ehrenm\"uller}
\address{Technische Universit\"at Hamburg-Harburg, Institut für Mathematik, 
Schwarzenberg\-str. 95, 21073 Hamburg, Germany}
\email{julia.ehrenmueller@tuhh.de}

\author[Yury Person]{Yury Person}
\address{Goethe-Universit\"at, Institut f\"ur Mathematik,
  Robert-Mayer-Str. 10, 60325 Frankfurt am Main, Germany}
\email{person@math.uni-frankfurt.de}

\author[Tuan Tran]{Tuan Tran}
\address{Freie Universit\"at Berlin, Institut f\"ur Mathematik,  
Arnimallee 3, 14195 Berlin, Germany}
\email{manhtuankhtn@gmail.com}

%%%%%%%%%%%%%%%%%%%%%%%%%%%%%%%%%%%%%%%%%%%%%%%%%%%%%%%%%

						% Abstract

%%%%%%%%%%%%%%%%%%%%%%%%%%%%%%%%%%%%%%%%%%%%%%%%%%%%%%%%%

\begin{abstract}
We consider biased $(1:b)$ Avoider-Enforcer games
in the monotone and strict versions.
%as introduced by Hefetz, Krivelevich, Stojaković, and Szabó \cite{HKSS10}.
In particular, we show that Avoider can keep his graph being a forest
for every but maybe the last round of the game
if $b \geq 200 n \ln n$. 
By this we obtain essentially optimal upper bounds on the threshold biases
for the non-planarity game, the non-$k$-colorability game,
and the $K_t$-minor game thus addressing a question and  
 improving the results of 
Hefetz, Krivelevich, Stojakovi\'c and Szab\'o. 
Moreover, we give a
slight improvement for the lower bound in 
the non-planarity game. 
\end{abstract}

\maketitle

%%%%%%%%%%%%%%%%%%%%%%%%%%%%%%%%%%%%%%%%%%%%%%%%%%%%%%%%%

						% Introduction

%%%%%%%%%%%%%%%%%%%%%%%%%%%%%%%%%%%%%%%%%%%%%%%%%%%%%%%%%

\section{Introduction}
Avoider-Enforcer games can be seen as the misère version of the well-known Maker-Breaker games 
(studied first by Lehman~\cite{L64}, Chv\'atal and Erd\H{o}s~\cite{CE78} and Beck~\cite{B82,B08}). 
 This means that, while playing according to their conventional rules, 
the players' goal is to lose the game. 
The general setting of Avoider-Enforcer games can be summarized as follows. 
Let $X$ be a finite set and let $\cF \subseteq 2^{X}$. 
The two players, called \emph{Avoider} and \emph{Enforcer}, alternately 
occupy a certain number of elements of the so-called \emph{board} $X$. 
The game ends when all elements are claimed by %one of
the players. Avoider wins if for every so-called \emph{losing set} $F \in \cF$, he does not occupy all 
elements of $F$ by the end of the game.  Otherwise Enforcer wins. In particular, 
it is not possible that the game ends in a draw. We may assume that Avoider is always the 
first player since the choice of the player who is making the first move does not have an impact on our results. 

\medskip

In the following we will focus on games where the board $X$ is given by the edge set $E(K_n)$ of a complete graph
 and $\cF_n$ is some graph property to be avoided. 
Following Hefetz, Krivelevich, Stojakovi\'c, and Szab\'o~\cite{HKSS10}, we consider 
two different versions of Avoider-Enforcer games. Let $b$ be a positive integer. 
In the original, \emph{strict} $(1:b)$ Avoider-Enforcer game 
(as investigated  e.g.\ by Beck~\cite{B02,B08}, Hefetz, Krivelevich, and Szab\'o~\cite{HKS07} and by 
Lu~\cite{L91, L92, L95}), Avoider occupies exactly $1$ and Enforcer exactly $b$ unclaimed edges per round. 
 If the number of unclaimed edges is strictly less than $b$ when it is Enforcer's turn, then he must select all 
the remaining unclaimed edges. For these strict rules, we define the 
lower threshold bias $f_{{\cF}_n}^-$ to be the largest integer such that 
Enforcer has a winning strategy for the $(1:b)$ game on $(E(K_n),{\cF}_n)$ 
for every $b\leq f_{{\cF}_n}^-$; and the upper threshold bias $f_{{\cF}_n}^+$ to be 
the smallest non-negative integer such that Avoider has a winning strategy for every $b>f_{{\cF}_n}^+$. 
In general, $f_{\cF_n}^-$ and $f_{\cF_n}^+$ do not coincide as shown 
by Hefetz, Krivelevich, and Szab\'o~\cite{HKS07}.  

\medskip

In the \emph{monotone} $(1:b)$ Avoider-Enforcer game,  
 Avoider occupies at least $1$ and Enforcer at 
least $b$ unclaimed edges per round. Again, if the number of 
unclaimed edges is strictly less than $b$ when it is Enforcer's turn, 
then he must select all the remaining unclaimed edges. 
Games with these monotone rules are bias monotone, as it was shown by Hefetz, Krivelevich, Stojakovi\'c, 
and Szab\'o in~\cite{HKSS10}. This means that there exists a unique 
 threshold bias $f_{{\cF}_n}^{mon}$ which is defined as the non-negative integer for which Enforcer wins
the monotone $(1:b)$ game if and only if $b\leq f_{{\cF}_n}^{mon}.$ 

\medskip

One might wonder at this point whether for any family ${\cF}_n$
there is some general relation between the three thresholds mentioned above 
like $ f_{{\cF}_n}^{-}\leq  f_{{\cF}_n}^{mon}\leq  f_{{\cF}_n}^{+}.$
Indeed, if ${\cF}_n={\cF}_{P_3,n}$ is the family of all paths on 3 vertices of $K_n$,
then these inequalities hold, as shown by Hefetz, Krivelevich, Stojakovi\'c, 
and Szab\'o in \cite{HKSS10}. However,
these inequalities are not true in general
and in fact the outcome of some Avoider-Enforcer games
in the strict setting can differ a lot from the outcome of the corresponding
monotone games. For instance, it was also shown in~\cite{HKSS10} and by Hefetz, Krivelevich, and Szab\'o in \cite{HKS07} that
for the Avoider-Enforcer connectivity game,
where  ${\cF}_n=C_n$ is the family of all spanning trees of $K_n$,
we have $f_{C_n}^{mon}=\frac{n}{\ln n}(1+o(1))$,
while $f_{C_n}^{+}=f_{C_n}^{-}=\lfloor \frac{n-1}{2}\rfloor.$

\medskip

In the present paper, we will be studying biased strict and monotone Avoider-Enforcer games, 
where Enforcer's goal is to maintain an (almost) acyclic graph. This will have 
a series of improvements on the bias of various games such as planarity, colorability 
and minor games. Before stating our results we survey the 
relevant developments so far. 

\medskip
Define $NC_n^k$ to be the set consisting of the edge sets of 
all non-$k$-colorable graphs on $n$ vertices. 
It was proved by Hefetz, Krivelevich, Stojakovi\'c, 
and Szab\'o~\cite{HKSS08} that for every $k\geq 3$, 
Avoider can win the strict $(1:b)$ \enquote{non-$k$-colorability} game 
$NC_n^k$ against any bias larger than $2kn^{1+\frac{1}{2k-3}}$. 
On the other hand, it was shown by the same authors~\cite{HKSS08} 
that there exists a constant $s_{k}$ such that Enforcer 
has a strategy to win the game for every $b \leq  s_{k} n$. 
Moreover, in the same paper the authors mention that 
there exists a constant $c>0$ such that $cn \leq f_{NC_n^2}^-  \leq f_{NC_n^2}^+\leq n^{3/2}$. 

\medskip

Let $M_n^t$ denote the set of all edge sets 
of all graphs on $n$ vertices containing a $K_t$-minor.  
Playing against a bias larger than $2n^{5/4}$, 
Avoider can win the strict $(1:b)$ $K_t$-minor game $M_n^t$ for 
every $t\geq 4$ whereas if $b$ is almost as large as $n/2$ Enforcer has a 
winning strategy where $t$ is some constant power of $n$, see Hefetz et~al.~\cite{HKSS08}. 
It was proved by Hefetz, Krivelevich, Stojakovi\'c, 
and Szab\'o in~\cite{HKSS10} that the threshold bias 
for the monotone version is of order $n^{3/2}$ for $t=3$.  

\medskip

Finally, let us introduce the \enquote{non-planarity} Avoider-Enforcer game. 
Let $NP_n$ be the set consisting of the edge sets of all non-planar graphs on 
$n$ vertices. In the so-called \enquote{non-planarity} game $NP_n$, 
Avoider's task is to keep his graph planar. Hefetz et al.\ proved in \cite{HKSS08}  
that in the strict $(1:b)$ non-planarity game, Avoider can succeed against any bias larger than
$2n^{5/4}$. Furthermore, their proof also can be applied when considering the monotone rules instead.

\medskip

The main results of our paper are the following two theorems. 
The first theorem gives a lower bound of $200n\ln n$ on the bias such 
that both in the monotone and in the strict $(1:b)$ Avoider-Enforcer game, 
Avoider can keep his graph acyclic apart from at most one unicyclic component. 
\begin{theorem}
\label{thm:forestplus}
For $n$ sufficiently large and $b \geq 200 n\ln n$, Avoider can ensure that both in the monotone and in the strict $(1:b)$ Avoider-Enforcer game
by the end of the game Avoider's graph is a forest plus at most one additional edge. 
\end{theorem}

In the strict $(1:b)$ game stated in the theorem below, 
 Avoider's task is to keep his graph acyclic for which he has again a 
winning strategy for some bias $b$ between $200n\ln n$ and $201n\ln n$.
\begin{theorem}
\label{thm:forest}
For $n$ sufficiently large, there is a bias $200n\ln n \leq b \leq 201n\ln n$
such that Avoider can ensure that in the strict $(1:b)$ Avoider-Enforcer game
by the end of the game Avoider's graph is a forest.
\end{theorem}

While these results are interesting in their own right, they can be applied directly 
   to three other games discussed above: the \enquote{non-$k$-colorability}, the \enquote{$K_t$-minor}, and 
the \enquote{non-planarity} Avoider-Enforcer games.

The two corollaries below are direct consequences of our main theorems above. In particular, 
these results improve upper bounds for $f_{NC_n^k}^+$ and $f_{NC_n^k}^{mon}$ with $k\geq 3$, 
and for $f_{NC_n^2}^-$. Furthermore better bounds are obtained for $f_{M_n^t}^+$ and $f_{M_n^t}^{mon}$ with $t \geq 4$ 
and for $f_{M_n^3}^-$. Finally, the bounds on $ f_{NP_n}^+$ and $ f_{NP_n}^{mon}$  are improved as well. 

\begin{corollary}
\label{cor:first}
For $n$ sufficiently large and $b \geq 200n\ln n$, Avoider can ensure that in the monotone/strict $(1:b)$ Avoider-Enforcer game
by the end of the game his graph is planar, $k$-colorable for $k\geq 3$, 
and does not contain a $K_t$-minor for $t\geq 4$.
Thus, 
$$f_{NP_n}^+, f_{NC_n^k}^+, \ f_{M_n^t}^+, f_{NP_n}^{mon}, f_{NC_n^k}^{mon}, \ f_{M_n^t}^{mon}\leq 200n\ln n.$$
\end{corollary}
\begin{proof}%[Corollary~\ref{cor:first}]
By Theorem~\ref{thm:forestplus}, Avoider can ensure that by the end of the game his graph is a forest plus at 
most one additional edge. Clearly, this graph is planar, 3-colorable, and does not contain a $K_4$-minor, proving the statement.
\end{proof}

\begin{corollary}
\label{cor:second}
For $n$ sufficiently large, there is a bias $200n\ln n \leq b \leq 201n\ln n$
such that Avoider can ensure that in the strict $(1:b)$ Avoider-Enforcer game
by the end of the game Avoider's graph is 2-colorable and does
not contain a $K_3$-minor. Thus, 
$$f_{NC_n^2}^-, f_{M_n^3}^-= O( n\ln n).$$
\end{corollary}
\begin{proof}%[Corollary~\ref{cor:second}]
By Theorem~\ref{thm:forest}, Avoider can ensure that by the end of the game his graph is a forest. 
Obviously, this graph is 2-colorable and does not contain a $K_3$-minor, proving the statement.
\end{proof}

Hefetz, Krivelevich, Stojakovi\'c, 
and Szab\'o conjectured in~\cite{HKSS08} that 
the Avoider-Enforcer non-planarity,
non-$k$-colorability and the $K_t$-minor games should be 
\emph{asymptotically monotone} as $n$ tends to infinity.
That is their upper and lower threshold should be of the same order, 
i.e.\ $f_{\cP_n}^-=\Theta( f_{\cP_n}^+ )$. 
Since in each of the three games we have lower bounds on 
$f_{\cP_n}^-$ that are linear in $n$, the Corollaries~\ref{cor:first} and~\ref{cor:second} 
 show that the threshold biases are at most $O(\ln n)$ factor apart, 
thus giving additional evidence that this conjecture might be true.

%In Corollary~\ref{cor:first} we strengthen this bound which is a direct consequence of Theorem~\ref{thm:forestplus}. 
%In Corollary~\ref{cor:first} we give a better upper bound for $f_{M_n^t}^+$ and $f_{M_n^t}^{mon}$ with $t \geq 4$, and in Corollary~\ref{cor:second} for $f_{M_n^3}^-$. 
%In Corollary~\ref{cor:first} we present an improved upper bound for $f_{NC_n^k}^+$ and $f_{NC_n^k}^{mon}$ with $k\geq 3$, and in Corollary~\ref{cor:second} for $f_{NC_n^2}^-$.  
%\medskip 
%\medskip
%\medskip 
%\medskip
\medskip

Coming back to the $(1:b)$ non-planarity Avoider-Enforcer game, it was also proved in~\cite{HKSS08} that in the strict version Enforcer can win whenever $b \leq \frac{n}{2} - o(n)$. 
Moreover, with a slight modification of the proof, the same result can be obtained
for the monotone rules.
We improve this bound as well. 

\begin{proposition}
\label{prop:lowerbound}
For $n$ sufficiently large and $b\leq 0.59n$, Enforcer
can ensure that both in the monotone and in the strict $(1:b)$ Avoider-Enforcer game, 
Avoider creates a non-planar graph.
Thus, $$0.59 n \leq f_{NP_n}^{mon},\;f_{NP_n}^- .$$
\end{proposition}

It should be mentioned that for the sake of readability, we do not optimize the constants in our theorems and proofs.
Our graph-theoretic notation is standard and follows \cite{D05}. In particular, given a graph $G$ its vertex set 
is denoted by $V(G)$ and its edge set by $E(G)$. The rest of the paper is organized as follows. In Section~$2$ 
we prove the two main results, namely Theorem~\ref{thm:forestplus} and Theorem~\ref{thm:forest}.  
%Then, in 
In Section~$3$ we study the non-planarity Avoider-Enforcer game 
and prove Proposition~\ref{prop:lowerbound}. Finally, in Section~$4$ we discuss some open problems.

%%%%%%%%%%%%%%%%%%%%%%%%%%%%%%%%%%%%%%%%%%%%%%%%%%%%%%%%%

						% Forests and almost-forests

%%%%%%%%%%%%%%%%%%%%%%%%%%%%%%%%%%%%%%%%%%%%%%%%%%%%%%%%%

\section{Forests and almost forests}

\begin{proof}[Proof of Theorem~\ref{thm:forestplus}]

Let $n$ be large enough and let $b\geq 200n\ln n$. In the following we will provide Avoider
with a strategy that ensures that by the end of the game Avoider's graph is a forest plus at 
most one additional edge.

Let $t$ be the  smallest integer with 
\begin{equation}\label{eq:t}
n \left(\frac{t+1}{10\ln n}\right)^t<3. 
\end{equation}
 %$\frac{n}{\ln^{t}n}<1$. 
 An easy calculation shows that $t=\Theta(\ln n)$, in particular, we have for large $n$ that 
\begin{equation}\label{eq:t_est}
 t<\ln n/3.
\end{equation}
%Note that $t=(1+o(1))\frac{\ln n}{\ln\ln n} < \ln n$. 
To succeed, Avoider will play according to $t$ stages in increasing order 
and each stage 
will last several rounds where it is possible that a stage lasts zero rounds.
In the first $t-1$ stages, Avoider always claims exactly one edge in each round,
connecting two components of his forest such that the sum of their sizes is minimal
 (whenever we talk about components, we mean the components of Avoider's forest).  
In the last stage, which will be shown to last at most one round, 
Avoider will claim an arbitrary further edge. We refer to edges, neither taken by Avoider 
nor by Enforcer, as \emph{unclaimed} edges.

Starting with Stage~$1$, Avoider plays according to the following rules. 
\medskip

\textbf{ Stage $k$ (for $k\in[t-1]$).}
If there exists an unclaimed edge $e$ between two components $T_1$ and $T_2$
with $|V(T_1)|+|V(T_2)|=k+1$, Avoider claims such an edge, thus creating
a component on the vertex set $V(T_1)\cup V(T_2)$. 
Then it is Enforcer's turn and the round is over.

Avoider is going to play according to Stage~$k$ 
in the next round as well. 
%Then he is done for this round
%and repeats Stage $k$ in the next round.
%\medskip
If there is no such edge $e$ to be claimed at Stage~$k$, Avoider proceeds with Stage $k+1$. 
(As mentioned above it might happen that there is no edge to be claimed at Stage~$k$ already when Avoider enters Stage $k$.
In that case, this stage lasts zero rounds, and Avoider immediately proceeds with Stage $k+1$.)

\medskip

\textbf{ Stage $t$.} In every further round, Avoider claims exactly
one arbitrary free edge.

\medskip

It is obvious that Avoider can follow the strategy. Moreover, it is easy to see
that as long as Avoider plays according to the strategy of the first $t-1$ stages,
his graph remains a forest. Thus, in order to show that the above described strategy
is indeed a winning strategy, it remains to show that the last stage 
lasts at most one round. However,  we prove the following claim first.

%\begin{claim}
%\label{claim:rec}
%Define a recurrence $C$ with 
%$$C(1):=1\ \text{ and }\ C(k):=k^2+\sum_{i+j\geq k \atop  i,j<k} ijC(i)C(j).$$
%Then Avoider creates at most $n_{k}:=C(k)n^{1-(k-1)\eps}$
%components of size $k$ throughout the game.
%\end{claim}

\begin{claim}
\label{claim:comp}
Let $k \leq t$ and let $n_k$ be the number of components of size exactly $k$
that Avoider creates when playing according to the strategy. Then 
\[ n_k \leq n \left(\frac{k}{10\ln n}\right)^{k-1}.\]
\end{claim}

\begin{proof}%[Claim~\ref{claim:comp}]
The claim is obviously true for $k= 1$. So, let $k>1$ and we proceed by induction. 
Observe that Avoider only creates components of size $k$ when he plays according to 
Stage~$k-1$. Thus, the number of such components is bounded from above
by the number of rounds that Stage $k-1$ lasts. 
When Avoider enters Stage~$k-1$ every existing component contains at most $k-1$ vertices 
and there are no unclaimed edges between two arbitrary components $T_1$ and $T_2$ with 
$|V(T_1)|+|V(T_2)| \leq k-1$. 
In particular, every unclaimed edge is either 
between two components $T_1$ and $T_2$ with $|V(T_1)|+|V(T_2)| \geq k$ or 
between two vertices within the same component which has size at most $k-1$.  
Obviously, the first case contributes at most $\sum_{1\le i \le j\le k-1\colon i+j\ge k}i j n_i n_j$
unclaimed edges. For the second case we find an upper bound of $(k-1)n$ by the following reason:
Let $n'_i$ denote the number of components of order $i$ after the end of Stage~$k-1$. 
Then the number of unclaimed edges within components after $k-1$ stages is at most 
$\sum_{i=1}^k \binom{i}{2}n'_i\le (k-1)\sum_{i=1}^k in'_i=(k-1)n $, since 
$\sum_{i=1}^k in'_i=n$.

Thus, at the beginning of Stage~$k-1$, the number of unclaimed edges is at most 
$\sum_{1\le i \le j\le k-1\colon i+j\ge k}i j n_i n_j + (k-1) n$.
%Since in each round at least $b$ edges are claimed,  
Since in each but possibly the last round at least $b+1$ edges are claimed ($1$ by Avoider and $b$ by Enforcer), 
we can %upper bound the number of 
%components of size $k$ created by Avoider as follows:
bound the number of components of size $k$ in Avoider's graph by
\begin{equation}\label{eq:nk_est}
n_k \le  \frac{1}{b+1} \left(\sum_{1 \le i \le j \le k-1\colon i+j\ge k} i j n_i  n_j + (k-1)n\right)+1.
\end{equation}
We use the induction hypothesis to estimate the sum $\sum_{1 \le i \le j\le k-1\colon i+j=s} i j n_i  n_j$ for $s=k$, \ldots, $2k-2$ as follows:
\begin{equation}\label{eq:ij_s}
 \sum_{\substack{1 \le i \le j \le k-1\\ i+j=s}} i j n_i  n_j\le  
\frac{n^2}{(10\ln n)^{s-2}}\sum_{\substack{1 \le i \le j \le k-1\\ i+j=s}} i^ij^j \le \frac{n^2}{(10\ln n)^{s-2}}\sum_{\substack{1 \le i \le j \le s-1\\ i+j=s}} i^ij^j.
\end{equation}
On the other hand, for $s \neq 6$ we have
\begin{multline}
\label{eq:ij_est}
\sum_{\substack{1 \le i \le j \le s-1\\ i+j=s}} i^ij^j < s^{s-1}+\sum_{2\le i \le s/2} i^is^{s-i} = s^{s-1} \left(1+s\sum_{2 \le i \le s/2} \left(\frac{i}{s}\right)^i\right) \\
\le s^{s-1} \left(1+s\sum_{2 \le i \le s/2} \left(\frac{2}{s}\right)^2\right)<3s^{s-1}. 
\end{multline}
For $s=6$, it is easy to check that 
\begin{equation}
\label{eq:s=6}
\sum_{\substack{1 \le i \le j \le s-1\\ i+j=s}} i^ij^j<3s^{s-1}.
\end{equation}
Therefore, we simplify~\eqref{eq:nk_est} using~\eqref{eq:ij_s},~\eqref{eq:ij_est},~\eqref{eq:s=6} and $b\ge 200n\ln n$ to
%and so by induction hypothesis and with $b\ge n\ln^3 n$ we obtain:

\begin{multline*}
n_k \le \frac{1}{200n\ln n} \left(\sum_{s=k}^{2k-2} 30n^2 \ln n \left(\frac{s}{10\ln n}\right)^{s-1}  + (k-1)n\right)+1\\
\le  \frac{3n}{20}\left(\frac{k}{10\ln n}\right)^{k-1}\sum_{s=k}^{2k-2}2^{k-s}+\frac{k-1}{200\ln n}+1\\
\overset{\eqref{eq:t_est}}{\le}   \frac{3n}{10}\left(\frac{k}{10\ln n}\right)^{k-1}+2\overset{\eqref{eq:t}}{\le} n\left(\frac{k}{10\ln n}\right)^{k-1},
\end{multline*}
where in the second inequality we estimated
\begin{multline*}
\left(\frac{s}{10\ln n}\right)^{s-1}=\left(\frac{k}{10 \ln n}\right)^{k-1} \prod_{i=1}^{s-k} \frac{k+i-1}{10 \ln n} \left(1+\frac{1}{k+i-1}\right)^{k+i-1} \\
\le \left(\frac{k}{10\ln n}\right)^{k-1} \left(\frac{2ke}{10 \ln n}\right)^{s-k} \overset{\eqref{eq:t_est}}{\le} \left(\frac{k}{10 \ln n}\right)^{k-1} 2^{k-s}.
\end{multline*}
This completes the proof of Claim \ref{claim:comp}.
%where the last inequality follows with $k\le t \le \frac{\ln n}{3}$, see~\eqref{eq:t}.
%since $k\le t=(1+o(1))\frac{\ln n}{\ln\ln n}$ and $n$ large.
\end{proof}

%\begin{proofc}[Claim~\ref{claim:rec}]
%The claim is obviously true for $k=1$.
%So, let $k>1$ and proceed by induction. 
%Observe that Avoider only creates
%components of size $k$ when he plays according
%to Stage $k-1.$ But when he enters Stage $k-1$
%every existing component contains at most $k-1$ vertices,
%and there is no free edge between two components $T_1$
%and $T_2$ with $|V(T_1)|+|V(T_2)|\leq k-1$ anymore.
%In particular, every free edge is either
%between two components $T_1$
%and $T_2$ with $|V(T_1)|+|V(T_2)|\geq k$
%or between two vertices of a same component.
%This yields an upper bound on the number of free edges
%as follows:
%\begin{align*}
%	\sum_{i+j\geq k \atop i,j<k} i\cdot j\cdot n_i\cdot n_j + k^2n
%= 	&\ \sum_{i+j\geq k \atop  i,j<k} ijC(i)C(j)\cdot n^{2-(i+j-2)\eps} + k^2n\\
%\leq 	& \left( \sum_{i+j\geq k\atop i,j<k} ijC(i)C(j)\right) n^{2-(k-2)\eps} + k^2n^{2-(k-2)\eps}\\
%=	& C(k)n^{2-(k-2)\eps},
%\end{align*} 
%for $k\leq \lceil 1/\eps \rceil.$
%Since in each round at least $b$ edges get claimed, Stage $k$ can last at most 
%$ C(k)n^{2-(k-2)\eps}/b = n_k$
%rounds, proving the upper bound on the number
%of components of size $k.$
%\end{proofc}

\medskip

Now, analogously to the calculation of the proof of Claim~\ref{claim:comp} it follows that, 
when Avoider enters the last stage, Stage~$t$, 
the number of remaining unclaimed edges is bounded by 
\begin{multline*}
 \sum_{\substack{1 \le i \le j \le t\\ i+j\ge t+1}} i j  n_i  n_j + t n\le  
 \sum_{s=t+1}^{2t} 30n^2 \ln n \left(\frac{t+1}{10\ln n}\right)^t 2^{t+1-s}  + tn\\
\overset{\eqref{eq:t}}{\le} 180 n \ln n+tn< 200n\ln n
\end{multline*}
%\[
%\sum_{i,j\in[t]\colon i+j\geq t +1} i j  n_i  n_j + t n\\
%< 	\frac{n}{\ln^{t} n} b < b,
%\]
by the choice of $t$ ($t<\ln n/3$) 
and for $n$ sufficiently large.
Thus, this last stage lasts at most one round.
\end{proof}

Now we turn to the case of the strict rule, when Enforcer has to claim 
exactly $b$ edges during each round (except possibly for the last one).

\begin{proof}[Proof of Theorem~\ref{thm:forest}]
We will show below that for large enough $n$, there exists $b$ 
with $200n\ln n\leq b\leq 201n\ln n$ and the remainder of $\binom{n}{2}$ divided by $b+1$ is 
at least $n\ln n$. 
%We claim that for large enough $n$ there is a bias 
%such that $$\text{rem}\left(\binom{n}{2},b+1\right)>n\ln^3 n.$$ 

Before proving this claim let us explain how the theorem follows then.
Let $b$ be given as above. Avoider now plays according to the same strategy
as given in the proof of Theorem~\ref{thm:forestplus} until he reaches Stage~$t$, 
where again $t$ is the smallest integer with $n \left(\frac{t+1}{10\ln n}\right)^t<3$. %$\frac{n}{\ln^{t} n}<1$. 
At this point, Avoider's graph is still a forest, the components of which are all
of size at most $t$.  
Analogously to the monotone case, there can be at most
$ t n < n\ln n /3$ unclaimed edges within components. 
However, since the remainder of the division $\binom{n}{2}/(b+1)$ is at least $n\ln n$, there exist unclaimed edges 
connecting two different components when Avoider enters Stage $t$ (provided $n$ is large enough). 
Now, Avoider just claims one such edge arbitrarily.
His graph remains a forest and afterwards, Enforcer must take all remaining edges.
Observe that in the case when Avoider is the second player, he does not even claim an edge
in the last round.

\medskip

So, it only remains to prove the above mentioned claim. Let $b_1=\lceil 200.5n\ln n\rceil.$
Moreover, let 
\[\binom{n}{2}=q_1(b_1+1)+r_1\ \text{ with }\ 0\leq r_1\leq b_1 \text{ and } q_1 \sim \frac{n}{401\ln n}.\]
If $r_1>n\ln n,$ we are done by setting $b=b_1$.
Otherwise, let
$b=b_1-\lceil 402 \ln^2 n\rceil.$
Then
$$\binom{n}{2}=q_1(b+1)+(r_1+q_1 \lceil 402 \ln^2 n\rceil).$$
Moreover, for large enough $n$, we obtain
$r_1+q_1 \lceil 402 \ln^2 n\rceil<b$, and therefore the remainder of the 
division $\binom{n}{2}$ by $(b+1)$ 
is at least $r_1+q_1 \lceil 402 \ln^2 n\rceil > n\ln n$, 
while $200n\ln n\leq b\leq 201n\ln n$.
\end{proof}

%%%%%%%%%%%%%%%%%%%%%%%%%%%%%%%%%%%%%%%%%%%%%%%%%%%%%%%%%

						% Lower bound in the non-planarity game

%%%%%%%%%%%%%%%%%%%%%%%%%%%%%%%%%%%%%%%%%%%%%%%%%%%%%%%%%

\section{Lower bound in the non-planarity game}
Before obtaining a lower bound for the non-planarity Avoider-Enforcer game in Proposition~\ref{prop:lowerbound}, 
we analyze another strict game where two players, the first player (denoted by FP) 
and the second player (denoted by SP), claim \emph{exactly} $1$ and $b$ edges, 
respectively.

\begin{proposition}\label{lower:proposition}
Let $c=1/1000$.
For $n$ sufficiently large and every $0.49n\leq b\leq 0.59n$
the second player in a strict $(1:b)$ game on %$K_n$ 
$E(K_n)$ can isolate at least
$$n - (1-c)\frac{n^2}{2b} \quad \text{ vertices,}$$
i.e.\ claim all edges that are incident to these vertices.

\end{proposition}

\begin{proof}
\textbf{Case 1. ($\mathbf{0.49n\leq b \leq 5n/9.}$)}
As long as there are at least 4 vertices not isolated by the second player (SP) and
not touched by the first player (FP), SP
can isolate a vertex in every fourth round.
Indeed, assume SP isolated a vertex in the previous round
and now wants to isolate one vertex within the next 4 rounds.
He fixes 4 vertices $v_1$, $v_2$, $v_3$, $v_4$ that are neither isolated
by him nor touched by FP. In a first round,
SP claims all edges between these 4 vertices and
at each $v_i$ he additionally claims $\lfloor (b-6)/4 \rfloor$ 
arbitrary incident edges. Now, it is FP's turn. He can touch at most 
one of these four vertices since all edges between them are already claimed by SP.
Without loss of generality, $v_1$, $v_2$, and $v_3$ are still untouched by FP.
Now in the second round SP claims at each of these three vertices $\lfloor b/3 \rfloor$ 
arbitrary incident edges. Again, FP can touch at most one of these three vertices at his turn. 
Without loss of generality, $v_1$ and $v_2$ are still untouched by FP after that. 
In the third round, SP claims  at each of these two vertices $\lfloor b/2 \rfloor$ arbitrary incident edges.
After FP's next turn, w.\,l.\,o.\,g.~$v_1$ is still untouced by FP. Now, SP simply claims all remaining incident edges at $v_1,$
which is possible since
$3+ \lfloor (b-6)/4 \rfloor + \lfloor b/3 \rfloor + \lfloor b/2 \rfloor + b > n,$
for large $n.$  Note that while SP isolates one vertex, FP can touch at most 8 other vertices.
It follows that the number of
vertices that SP isolates in total is at least $\lfloor n/9 \rfloor\geq n - (1-c)\frac{n^2}{2b}$.

\medskip

\textbf{Case 2. ($\mathbf{5n/9\leq b \leq 11n/19.}$)}
Analogously to Case 1, SP
can isolate a vertex in every third round as long as there are at least 3 vertices
not touched by FP. This time,
SP starts by only fixing three vertices $v_1,v_2,v_3$ and 
isolates then one of them within three rounds,
which is possible since 
$2 + \lfloor (b-3)/3 \rfloor + \lfloor b/2 \rfloor + b > n,$
for large $n.$ 
It follows then that SP isolates 
at least $\lfloor n/7 \rfloor\geq n - (1-c)\frac{n^2}{2b}$ vertices in total.

\medskip

\textbf{Case 3. ($\mathbf{11n/19\leq b \leq 0.59n.}$)}
Analogously to Case 2, SP
can isolate a vertex in every third round as long as there are at least 3 vertices
not touched by FP. 
In a first phase, SP follows the above described strategy and he isolates 
$n-1.5b$ vertices, which happens in at most $3n-4.5b$ rounds.
During this phase, 
FP can touch at most $6n-9b$ vertices.
Afterwards, for every vertex that is
neither isolated by SP nor touched by FP,
SP only needs to claim at most $1.5b$ further incident edges
in order to isolate it.  
But then, analogously to the previous cases, SP can isolate 
one vertex in every second round,
since 
$1+\lfloor (b-1)/2 \rfloor + b \geq 1.5b.$
Thus, in the second phase after at most  $3n-4.5b$ rounds, 
SP isolates a vertex in every second round 
as long as possible.
Since at the beginning of the second phase
at least $n-(n-1.5b)-(6n-9b)=10.5b-6n$
vertices were neither isolated by SP nor touched by FP,
SP can isolate at least $(10.5b-6n)/5$
further vertices.
In total SP will isolate at least 
$(n-1.5b)+(10.5b-6n)/5\geq n - (1-c)\frac{n^2}{2b}$
 vertices.
\end{proof}

\begin{lemma}
\label{lem:nonplanarstrict}
For $n$ sufficiently large and $b\leq 0.59n$ Enforcer
can ensure that in the strict $(1:b)$ game on %$K_n$
$E(K_n)$ Avoider creates a non-planar graph. Thus, $$0.59 n \leq f_{NP_n}^-.$$
\end{lemma}

\begin{proof}
Since the statement is already proved for $b\leq 0.49n$ in \cite{HKSS08},
we may assume from now on that $0.49n\leq b.$ The following 
proof will be a slight modification of the one given in \cite{HKSS08}.
Let %$0<c<1/1000$ 
$c=1/1000$ be as in Proposition~\ref{lower:proposition}
and choose an integer $k\geq 3$ such that
\begin{align}\label{constant} 
\frac{k}{k-2}\left(1-\frac{c}{2}\right)<1.
\end{align}

Enforcer's strategy consists of two goals: First of all, he wants to prevent Avoider from creating cycles of length at most $k.$ Secondly, he wants to isolate a large number of vertices to ensure that
Avoider's graph lives on a small vertex set. For this he splits his bias $b=b_1+b_2$
($b_1$ and $b_2$ will be chosen later) and uses $b_1$ for his first goal, 
and $b_2$ for the second goal.

\medskip

\textbf{Preventing cycles.} It follows from the work of 
Bednarska and \L{}uczak~\cite{BL00} (see also the proof of Theorem~2.3 in~\cite{HKSS08}), that  
for every $3\leq i\leq k$ there is a constant $c_i$
such that, for sufficiently large $n$, Enforcer can prevent Avoider from
claiming a cycle of length $i$ if Enforcer is allowed to claim at least
$c_i n^{\frac{i-2}{i-1}}$ edges. 
Let $C=\max\{c_i:\ 3\leq i\leq k\}.$
Then, simultaneously playing
according to the different strategies for preventing cycles of length
$3\leq i\leq k,$
Enforcer can ensure that Avoider's graph has girth larger than $k$
if he claims at least
\[\sum _{i=3}^k c_i n^{\frac{i-2}{i-1}}\leq Ck n^{\frac{k-2}{k-1}}=:b_1\]
edges per round. Observe that $b_1=o(b)$.

\medskip

\textbf{Isolating vertices.} Let $b_2=b-b_1=b(1-o(1)).$ In each round
Enforcer uses $b_2$ edges to play according to the strategy given in the proof of 
Proposition \ref{lower:proposition}. Therefore, he isolates at least
$n - (1-c)\frac{n^2}{2b_2} \geq n - \left( 1 - \frac{c}{2} \right) \frac{n^2}{2b}$
vertices.
\medskip

Now, let Enforcer split his bias $b=b_1+b_2$, and thus play so as to prevent 
cycles of length at most $k$, while at the same time to isolate at least 
$n - \left( 1 - \frac{c}{2} \right) \frac{n^2}{2b}$ vertices. Notice, that it does not hurt Enforcer 
if the combination of the above strategies leads to claiming the same edge more than once - 
Enforcer can claim an arbitrary edge instead since this does not destroy the properties of the graph he is about 
to create.  Let $A$ be Avoider's graph at the end of the game.
We know that $|V(A)|\leq \left( 1 - \frac{c}{2} \right) \frac{n^2}{2b}$
and $\text{girth}(A)>k$. If $A$ was planar, then, by a standard application
of Euler's formula, we would have
$$|E(A)|<\frac{k}{k-2}\left( |V(A)|-2 \right)
	<\frac{k}{k-2} \left( 1 - \frac{c}{2} \right) \frac{n^2}{2b}.$$ 
However, by the number of rounds the game lasts, we have
$$|E(A)|\geq \left\lfloor \frac{\binom{n}{2}}{b+1} \right\rfloor
	>\frac{k}{k-2} \left( 1 - \frac{c}{2} \right) \frac{n^2}{2b},$$
using (\ref{constant}), for $n$ sufficiently large. Thus, Avoider's graph is
non-planar and Enforcer wins.
\end{proof}

\begin{lemma}
\label{lem:nonplanarmon}
For $n$ sufficiently large and $0.49n\leq b\leq 0.59n$ Enforcer
can ensure that  Avoider creates a non-planar graph in the monotone $(1:b)$ game on %$K_n$ 
 $E(K_n)$.
Thus, $$0.59 n \leq f_{NP_n}^{mon}.$$
\end{lemma}

\begin{proof}
Let $A$ be Avoider's graph throughout the game,
and let $A^*\subseteq A$ be a subgraph consisting
of exactly one edge from every round played so far.
Enforcer claims in every round exactly $b$ edges
according to the strategy given in the proof of the previous lemma,
assuming $A^*$ to be Avoider's graph. 
If this strategy asks Enforcer
to claim an edge from $A\setminus A^*$, he will claim another
arbitrary edge instead. We distinguish two cases.

\medskip

{\bf Case 1. $\mathbf{|V(A)|> 3n}$.} Then, by Euler's formula,
Avoider's graph is non-planar and Enforcer wins.

\medskip

{\bf Case 2. $\mathbf{|V(A)|\leq 3n}$.} Then
the number of rounds the game lasts is at least
$\frac{\binom{n}{2}-3n}{b}= \frac{n^2}{2b}(1-o(1)),$
which also gives 
$$|E(A^*)|\geq \frac{n^2}{2b}(1-o(1)).$$
By the above described strategy we get again, similar to the proof of Lemma~\ref{lem:nonplanarstrict}, 
$|V(A^*)|\leq \left( 1 - \frac{c}{2} \right) \frac{n^2}{2b}$
as well as $\text{girth}(A^*)>k,$ ensuring that $A^*$ cannot be planar provided that $n$ is large enough.
\end{proof}

%\begin{theorem}
%The following bounds hold for large enough $n$.
%
%$$0.59 n \leq f_{NP_n}^-  \leq f_{NP_n}^+\leq n^{1+o(1)}.$$

%Moreover, the same bounds hold for the monotone thresholds.
%\end{theorem}

\begin{proof}[Proof of Proposition~\ref{prop:lowerbound}]
This proposition follows directly from Lemma~\ref{lem:nonplanarstrict} and Lemma~\ref{lem:nonplanarmon}. 
\end{proof}

\section{Open questions}
For each of the games considered for Corollary \ref{cor:first}, we have shown that the lower and upper threshold
bias differ at most by a factor of $\ln n$. However, we believe that this factor
can be replaced by some constant. We wonder whether this can already be done for the strategy we analyzed in the proof of Theorem~\ref{thm:forestplus}, where we have shown that Avoider can keep his graph almost acyclic.   

\begin{question}
Is there a constant $C>0$ such that the following holds: For $n$ sufficiently large
and $b\geq Cn$, Avoider has a strategy that creates at most one cycle
in the monotone/strict $(1:b)$ game? 
\end{question}

In case the question above can be answered positively, the following conjecture would follow immediately. %(for $k\geq 3$ and $t\geq 4$).

\begin{conjecture}[\cite{HKSS08}]
The Avoider-Enforcer non-planarity, non-k-colorability and $K_t$-minor games are asymptotically monotone for every $k\ge 3$ and $t \ge 4$. 
\end{conjecture}

Our result on the lower threshold bias for the non-planarity game is obtained
by splitting Enforcer's strategy into two parts. The first part,  
based on the strategy from \cite{HKSS08}, is
to prevent small cycles in Avoider's graph. The second part is to isolate a large number
of vertices. So, our improvement was obtained by studying a positional game
in which one player has the goal to isolate as many vertices as possible. This game itself seems to be of interest.

\begin{question}
Let $b\in \mathbb{N}$. What is the largest number of vertices that the second player can isolate in a 
$(1:b)$ game on $E(K_n)$ under the strict rules?
\end{question}

%%%%%%%%%%%%%%%%%%%%%%%%%%%%%%%%%%%%%%%%%%%%%%%%%%%%%%%%%

						% Acknowledgement

%%%%%%%%%%%%%%%%%%%%%%%%%%%%%%%%%%%%%%%%%%%%%%%%%%%%%%%%%

\subsection*{Acknowledgement}
We are grateful to Milo\v{s} Stojakovi\'{c} for helping us to further improve our bounds in Theorem~\ref{thm:forestplus} and Theorem~\ref{thm:forest}.

%%%%%%%%%%%%%%%%%%%%%%%%%%%%%%%%%%%%%%%%%%%%%%%%%%%%%%%%%

						% References

%%%%%%%%%%%%%%%%%%%%%%%%%%%%%%%%%%%%%%%%%%%%%%%%%%%%%%%%%

\bibliographystyle{amsplain}
\bibliography{references}

\providecommand{\bysame}{\leavevmode\hbox to3em{\hrulefill}\thinspace}
\providecommand{\MR}{\relax\ifhmode\unskip\space\fi MR }
\providecommand{\MRhref}[2]{%
  \href{http://www.ams.org/mathscinet-getitem?mr=#1}{#2}
}
\providecommand{\href}[2]{#2}
\begin{thebibliography}{10}

\bibitem{B82}
J.~Beck, \emph{Remarks on positional games. {I}}, Acta Math. Acad. Sci. Hungar.
  \textbf{40} (1982), no.~1-2, 65--71.

\bibitem{B02}
\bysame, \emph{Ramsey games}, Discrete Math. \textbf{249} (2002), no.~1-3,
  3--30, Combinatorics, graph theory and computing (Louisville, KY, 1999).

\bibitem{B08}
\bysame, \emph{Combinatorial games}, Encyclopedia of Mathematics and its
  Applications, vol. 114, Cambridge University Press, Cambridge, 2008,
  Tic-tac-toe theory.

\bibitem{BL00}
M.~Bednarska and T.~{\L}uczak, \emph{Biased positional games for which random
  strategies are nearly optimal}, Combinatorica \textbf{20} (2000), no.~4,
  477--488.

\bibitem{D05}
B.~Bollob{\'a}s, \emph{Modern graph theory}, Graduate Texts in Mathematics,
  vol. 184, Springer-Verlag, New York, 1998.

\bibitem{CE78}
V.~Chv{\'a}tal and P.~Erd{\H{o}}s, \emph{Biased positional games}, Ann.
  Discrete Math. \textbf{2} (1978), 221--229, Algorithmic aspects of
  combinatorics (Conf., Vancouver Island, B.C., 1976).

\bibitem{HKSS08}
D.~Hefetz, M.~Krivelevich, M.~Stojakovi{\'c}, and T.~Szab{\'o},
  \emph{Planarity, colorability, and minor games}, SIAM J. Discrete Math.
  \textbf{22} (2008), no.~1, 194--212.

\bibitem{HKSS10}
\bysame, \emph{Avoider-{E}nforcer: the rules of the game}, J. Combin. Theory
  Ser. A \textbf{117} (2010), no.~2, 152--163.

\bibitem{HKS07}
D.~Hefetz, M.~Krivelevich, and T.~Szab{\'o}, \emph{Avoider-{E}nforcer games},
  J. Combin. Theory Ser. A \textbf{114} (2007), no.~5, 840--853.

\bibitem{L64}
A.~Lehman, \emph{A solution of the {S}hannon switching game}, J. Soc. Indust.
  Appl. Math. \textbf{12} (1964), 687--725.

\bibitem{L91}
X.~Lu, \emph{A matching game}, Discrete Math. \textbf{94} (1991), no.~3,
  199--207.

\bibitem{L92}
\bysame, \emph{Hamiltonian games}, J. Combin. Theory Ser. B \textbf{55} (1992),
  no.~1, 18--32.

\bibitem{L95}
\bysame, \emph{A {H}amiltonian game on {$K_{n,n}$}}, Discrete Math.
  \textbf{142} (1995), no.~1-3, 185--191.

\end{thebibliography}

\end{document}